\documentclass[12pt]{amsart}
\usepackage{wrapfig}
\usepackage{graphicx}
\usepackage[all]{xy}
\usepackage{amsmath,amsthm,amssymb}
\usepackage{mathabx,epsfig}

\usepackage[margin=1in]{geometry}
\usepackage{amsmath,amscd}
\usepackage{mathtools}

\usepackage[all]{xy}

\usepackage{thmtools}
\usepackage{thm-restate}
\usepackage{hyperref}
\usepackage{cleveref}

\newtheorem{theorem}{Theorem}
\newtheorem{conjecture}[theorem]{Conjecture}

\newtheorem{lemma}{Lemma}[theorem]
\newtheorem{proposition}{Proposition}[theorem]
\usepackage[small,compact]{titlesec}
\usepackage{enumerate}
\usepackage{times}
\makeatletter

\newcommand{\Rmnum}[1]{\expandafter\@slowromancap\romannumeral #1@}
\makeatother

   \newcommand{\Hom}{\operatorname{Hom}}
 \newcommand{\Pic}{\operatorname{Pic}}

 \newcommand{\PrePer}{\operatorname{PrePer}}

 \newcommand{\Disc}{\operatorname{Disc}}
 \newcommand{\Tor}{\operatorname{Tor}}
 \newcommand{\Norm}{\operatorname{Norm}}
  
  \newcommand{\Div}{\operatorname{Div}}
 \newcommand{\Jac}{\operatorname{Jac}}
  
 \newcommand{\Aut}{\operatorname{Aut}}

\newcommand{\Gal}{\operatorname{Gal}}
\newcommand{\rank}{\operatorname{rank}}

\usepackage{xcolor}
\usepackage{titlesec}
\titleformat{\section}[block]{\color{black}\large\filcenter}{}{1em}{}
\titleformat{\subsection}[hang]{\bfseries}{}{1em}{}
\theoremstyle{remark}

\newtheorem{remark}{Remark}[theorem]
\begin{document}
\title{Rational Points on certain families of symmetric equations}
\author{Wade Hindes}
\date{\today}
\renewcommand{\thefootnote}{}
\maketitle
\begin{abstract} We generalize the work of Dem'janenko and Silverman for the Fermat quartics, effectively determining the rational points on the curves $x^{2m}+ax^m+ay^m+y^{2m}=b$ whenever the ranks of some companion hyperelliptic Jacobians are at most one. As an application, we explicitly describe $X_d(\mathbb{Q})$ for certain $d\geq3$, where $X_d: T_d(x)+T_d(y)=1$ and $T_d$ is the monic Chebychev polynomial of degree $d$. Moreover, we show how this later problem relates to orbit intersection problems in dynamics. Finally, we construct a new family of genus $3$ curves which break the Hasse principle, assuming the parity conjecture, by specifying our results to quadratic twists of $x^4-4x^2-4y^2+y^4=-6$.            
\end{abstract}
\section{1. Introduction}{\label{Intro}}

\indent \indent Although there has been considerable progress in recent years, effectively determining the complete set of rational points on a curve remains a difficult task. This task becomes increasingly difficult if one wishes to study the rational points in families of curves. However, understanding the behavior of rational points in families provides a good testing ground for conjectures on curves in general, making it a useful enterprise. 

One rare example of a family of curves where one can explicitly bound the set of rational points in many circumstances is the Fermat quartic $F_b: x^4+y^4=b$, which was studied in detail by Dem'janenko \cite{Dem}. He noticed that $F_b$ possesses two maps $\phi_1(x,y)=(x^2,xy)$ and $\phi_2(x,y)=(y^2,xy)$ to the elliptic curve $E_b:y^2=x\cdot(x^2-b)$, from which he could deduce that for all $4$-th power free rational numbers $b\geq3$, if $\rank\big(E_b(\mathbb{Q})\big)\leq1$, then $F_b(\mathbb{Q})$ is empty; see \cite{Dem} for details. 

However, after noticing that it is the symmetry (in $x$ and $y$) of $F_b$ that equips us with multiple mappings, we expand the class of curves on which Dem'janenko's technique may be applied. As a first example, let $K/\mathbb{Q}$ be a number field and let $a$ and $b$ be in $K$. We define the curves 
\begin{equation}{\label{Quartic}} F=F_{(a,b)}: x^4+ax^2+ay^2+y^4=b\;\;\;\;\;\text{and}\;\;\;\;\ E=E_{(a,b)}: y^2=x\cdot\big(x^2-4ax-(16b+4a^2)\big).
\end{equation} 
 One checks that both $F$ and $E$ are nonsingular when \[\Delta_F=\Delta_{(a,b)}:=b\cdot(a^2+2b)\cdot(a^2+4b)\neq0,\] and that we have two maps $\phi_1,\phi_2: F_{(a,b)}\rightarrow E_{(a,b)}$ given by 
 \begin{equation}{\label{maps}}\phi_1(x,y)=\left(-4x^2,x\cdot(8y^2+4a)\right)\;\;\;\;\;\text{and}\;\;\;\;\;\phi_2(x,y)=\left(-4y^2,y\cdot(8x^2+4a)\right).\end{equation}
Similarly, for $\alpha\in K^*/(K^*)^2$, let $F^{(\alpha)}=F_{(\alpha\cdot a,\alpha^2\cdot b)}$ and $E^{(\alpha)}=E_{(\alpha\cdot a,\alpha^2\cdot b)}$ be the quadratic twists of $F$ and $E$ by $\alpha$ respectively. As an application of Dem'janenko's framework, we prove the following theorem. 
\begin{restatable}{thm}{Hasse}
\label{thm:Hasse} Let $F$ and $E$ be the symmetric quartic and elliptic curves corresponding to $(a,b)=(-4,-6)$. Then all of the following statements hold:  
\begin{center}\begin{enumerate}
\item If $p>0$ is a prime such that $p\equiv1\bmod{24}$, then $F^{(p)}$ is everywhere locally solvable.   
\item The global root number $W(E^{(p)})=-1$ for all positive, odd primes. 
\item If $\alpha >7\cdot10^{74}$ is square-free and $\rank\left(E^{(\alpha)}(\mathbb{Q})\right)\leq1$, then $F^{(\alpha)}(\mathbb{Q})=\varnothing$.
\item If $p\notequiv\pm{1}\bmod{16}$, then $\rank(E^{(p)})\leq2$. 
\item Assuming the parity conjecture, if $p>3\cdot10^{74}$ and $p\equiv25\bmod{48}$, then $F^{(p)}(\mathbb{Q})=\varnothing$, and $F^{(p)}$ breaks the Hasse principle.
\end{enumerate}
\end{center}
\end{restatable} 
 
 Moreover, in Section \ref{Orbit}, we indicate how these symmetric quartic equations appear in the study of certain orbit intersection problems in quadratic dynamics. In particular, we use Dem'janenko's technique to prove results related to the rational points on the Chebychev curve \[X_d: T_d(x)+T_d(y)=1,\] discussed in the abstract. Interestingly, $X_d$ has several small rational points for most values of $d$. As a preliminary result, we have the following theorem. 
\begin{restatable}{thm}{Chebychev}
\label{thm:Chebychev}
Let $T_d$ denote the Chebychev polynomial of degree $d$ and let $X_{d}: T_{d}(x)+T_{d}(y)=1$. We describe $X_d(\mathbb{Q})$ explicitly in the following cases: 
\begin{enumerate}
\item If $d\equiv0\bmod{3}$, then $X_d(\mathbb{Q})=\varnothing$.  
 \item If $d\notequiv0\bmod{3}$ and $d\equiv0\bmod{4}$, then $X_{d}(\mathbb{Q})=\{(0,\pm{1}),(\pm{2},\pm{1}),(\pm{1},0),(\pm{1},\pm{2})\}$.
 \item Suppose that $d\notequiv0\bmod{3}$ and $d\equiv0\bmod{5}$.\\
 \indent (a). If $d$ is odd, then $X_d(\mathbb{Q})=\{(0,1),(1,0),(-1,2), (2,-1)\}$.\\
 \indent (b). If $d\equiv2\bmod{4}$, then $X_d(\mathbb{Q})=\{(\pm{1},\pm{2}),(\pm{2},\pm{1})\}$.   
  \end{enumerate}
\end{restatable}  
In fact, evidence suggests that the complete list of rational points depends only on the congruence class of $d\bmod{12}$, which we state in the following conjecture. 
\begin{conjecture}{\label{conjecture}} If $P=(x,y)\in X_d(\mathbb{Q})$ and $d\geq3$, then $x,y\in\{0,\pm{1},\pm{2}\}$. In particular, $\#X_d(\mathbb{Q})=0,4,8$ or $12$ and $X_d(\mathbb{Q})$ can be described explicitly in terms of the congruence class of $d\bmod{12}$. 
\end{conjecture}
Dem'janenko's work has been generalized by Manin. In particular, let $C$ be a curve, $A$ be an abelian variety, and $\{\phi_1,\phi_2,\dots\phi_n\}$ be maps $\phi_i:C\rightarrow A$. If the $\phi_i$ are independent elements of the $\mathbb{Z}$-module $\Hom(C,A)$, then one can effectively determine the rational points on $C$, provided that $\rank(A(K))\leq n-1$; see \cite[Theorem 13.3.1]{Coh}. Moreover, Silverman has shown that if we have a finite Galois-module acting on both $C$ and $A$ in a way that commutes with the maps $\phi_i$, then we can bound the rational points on twists of $C$ in terms of the rank of the twists of $A$; see \cite[Theorem 1]{Silv} for a precise statement.  

However, even in the $n=2$ case, it can be difficult to verify this independence. For instance, one usually needs to compute the degree of $(\phi_1+\phi_2)^*D$, subject to some choice of symmetric ample divisor $D$ on the abelian variety; for example, see \cite[Proof of Theorem 1]{Silv}. 

However, since the group law on abelian varieties of large dimension is normally quite complicated, this computation seems less than feasible. Nonetheless, we show how one can circumvent this degree calculation, subject to some constraints. As an application, we further expand the class of curves on which Dem'janenko's technique may be applied by introducing some higher degree symmetric equations analogous to those defined in (\ref{Quartic}) above. 

Specifically, let $m\geq3$ be an odd integer and define \[C=C_{(a,b)}: X^{2m}+aX^m+aY^m+Y^{2m}=b\;\;\;\text{and}\;\;B=B_{(a,b)}:y^2=-x^{2m}-ax^{m}+(a^2/4+b).\] Furthermore, let $\infty_+$ and $\infty_-$ be the points at infinity on $B$ (possibly defined over a quadratic extension) and let $J=J_{(a,b)}$ be the Jacobian of $B$, which we identify with $\Pic^0(B)$. As in the quartic case, we have maps $\phi_1,\phi_2:C\rightarrow J$ defined over $K$, given by \[\phi_1(x,y)=\big[2\cdot\big(x\;,\;y^m+a/2\big)-(\infty_++\infty_-)\big],\;\;\text{and}\;\;\;\phi_2(x,y)=\big[2\cdot\big(y\;,\;x^m+a/2\big)-(\infty_++\infty_-)\big].\]       

Similarly, set $C_\alpha=C_{(\alpha\cdot a,\alpha^2\cdot b)}$ and $B_\alpha=B_{(\alpha\cdot a,\alpha^2\cdot b)}$ for all $\alpha\in K^*/(K^*)^m$. After adapting the argument in the rank one case of \cite[Theorem 1]{Silv}, we establish the following result.

\begin{restatable}{thm}{Symm}
\label{thm:Symm} For all but finitely many $\alpha\in K^*/(K^*)^m$, if $\rank\left(J_\alpha(K)\right)\leq1$, then $C_\alpha$ has no $K$-affine points and $C_\alpha(K)\subseteq\big\{[1,\zeta,0]\;\big\vert\;\zeta^{2m}+1=0\}$. In particular, if $C_{\alpha}$ is defined over the rational numbers, then $C_\alpha\big(\mathbb{Q}\big)=\varnothing$.  
\end{restatable}

To prove our theorems, since we are unable to do the degree calculation mentioned above, we develop an alternative method that avoids this calculation. We proceed with this task in Section \ref{Symm}.   

\section{2. Symmetric Diophantine Problems}{\label{Symm}} 

\indent \indent We begin by fixing some notation. Let $K/\mathbb{Q}$ be a numberfield and $V_{/K}$ be a variety defined over $K$. For finite $\Gal(\bar{K}/K)$-modules $G$ that admit a Galois equivariant  embedding into $\Aut(V)$, we define \[V[G]:=\{P\in V(\bar{K})\vert\;\gamma(P)=P\;\text{for some}\;\gamma\in G,\;\gamma\neq1\}.\]  
Moreover, for $\xi\in H^1(K,G)$, let $K_V(\xi)$ be the \emph{splitting field of} $\xi$ in $H^1\big(K,\Aut(V)\big)$, that is, $K_V(\xi)$ is the smallest extension of $K$ such that $\xi$ is trivial in $H^1\big(K_V(\xi),\Aut(V)\big)$. Finally, for an abelian variety $A$, let $\Aut^*(A)$ be the set of automorphisms of $A$ (as both a group and a variety). 

Every cohomology class $\xi\in H^1(K,G)$ gives rise to a twist $V_\xi$ of $V$ which is defined over $K$. Moreover, if $G$ admits a Galois equivariant embedding into $\Aut^*(A)$, then $A_\xi$ is an abelian variety. (See, for example, \cite[X.2]{SilvElliptic}.) With these definitions in place, we prove the following theorem.            
\begin{restatable}{thm}{General}
\label{thm:General}
 Suppose that $C\subset\mathbb{P}^n$ is a curve, $A$ is an abelian variety, and $N$ is an integer such that all of the following conditions are satisfied: 
\begin{enumerate}
\item There is a finite $\Gal(\bar{K}/K)$-module $G$ together with a Galois equivariant embedding into both $\Aut(C)$ and $\Aut^*(A)$ such that $A[G]\subseteq A[N]$.
\item There exists $\lambda\in PGL_{n+1}(K)$ such that $\lambda(C)\subseteq C$, and $\lambda$ commutes with the action of $G$ on $C$.  
\item There is a map $\phi:C\rightarrow A$ which commutes with the action of $G$ on $C$ and $A$.
\item The maps $(\phi\circ\lambda|_C)+\phi:C\rightarrow A$ and $(\phi\circ\lambda|_C)-\phi:C\rightarrow A$ are non-constant. 
\end{enumerate}   
Then for almost all $\xi\in H^1(K,G)$, if $\rank(A_{\xi}(K))\leq1$, then $C_{\xi}(K)\subseteq C_{\xi}[G]$.   
\end{restatable} 
\begin{proof}For convenience, let $\phi_1:=\phi$ and $\phi_2:=(\phi\circ\lambda)$. Choose a $\bar{K}$ isomorphism $i_\xi:C_\xi\rightarrow C$ such that the cocycle $\sigma\rightarrow i_\xi^{\sigma}\circ i_\xi^{-1}=\Xi_{\sigma}$ in $G$ represents the cohomology class $\xi$. Similarly, choose a $\bar{K}$-isomorphism $I_{\xi}:A_\xi\rightarrow A$ such that $I_\xi^\sigma\circ I_\xi^{-1}=\Xi_\sigma$ for all $\sigma\in\Gal(\bar{K}/K)$.

By assumptions $2$ and $3$, both $\phi_1$ and $\phi_2$ are defined over $K$ and commute with the action of $G$ on $C$ and $A$. It follows that the maps $\phi_i^{(\xi)}:=(I_\xi^{-1}\circ\phi_i\circ i_\xi):C_\xi\rightarrow A_\xi$ are defined over $K$. With the setup in place, we have the following lemma: 
\begin{lemma}{\label{Lemma1}} For almost all $\xi\in H^1(K,G)$, if $P_\xi\in C_\xi(K)$ and $\rank\big(A_\xi(K)\big)\leq1$, then either \begin{equation}{\label{lemma equation}}\phi_1^{(\xi)}(P_{\xi})+\phi_2^{(\xi)}(P_{\xi})\in A_\xi(K)_{\Tor}\;\;\;\; \text{or} \;\;\;\;\phi_1^{(\xi)}(P_{\xi})-\phi_2^{(\xi)}(P_{\xi})\in A_\xi(K)_{\Tor}.\end{equation}  
\end{lemma}
\begin{proof}[Proof of Lemma] If $P_\xi\in C_\xi(K)$ and $\rank\big(A_\xi(K)\big)\leq1$, then we may write $\phi_i^{(\xi)}(P_\xi)=n_i^{(\xi)}\cdot S_\xi+T_i^{(\xi)}$, where $S_\xi$ generates the free part of  $A_\xi(K)$ and the $T_i^{(\xi)}$ are torsion points. Let $P=i_\xi(P_\xi)$ and let $S$ and $T_i$ denote the images of $S_\xi$ and $T_i^{(\xi)}$ under $I_\xi$ respectively. In particular, we see that $\phi_i(P)=n_i^{(\xi)}S+T_i$.  

Fix an ample symmetric divisor $D\in\Div_K(A)$, an associated Weil height function $h_D$ \big(determined up to $O(1)$\big), and canonical height function $\hat{h}_D$. Moreover, we take $h_C$ to be the Weil height given by our fixed embedding of $C\subset\mathbb{P}^{n}$. Since $\lambda\in PGL_{n+1}$ is a linear map, we have that 
\begin{equation}{\label{height}} h_D\big(\phi_1(Q)\big)=h_D\big(\phi_2(Q)\big)+O(1),\;\;\;\text{for all}\;\;\;Q\in C(\bar{K}).\end{equation} To see this, we use the fact that the height function $h_D$ is given by some morphism $\pi:A\rightarrow\mathbb{P}^m$. Then we have rational maps $\mathbb{P}^n\rightarrow\mathbb{P}^m$, given by a tuple of some degree $d$ homogeneous polynomials that restrict to the morphisms $(\pi\circ\phi)$ and $(\pi\circ(\phi\circ\lambda))$ on $C\rightarrow\mathbb{P}^m$. It follows from \cite[Theorem B.2.5]{Silv-Hindry}, that $h_D(\phi_i(Q))=d\cdot h_C(Q)+O(1)$ for all $Q\in C(\bar{K})$, since $\lambda$ is a linear map (hence does not alter the degree). In particular, $h_D\big(\phi_1(Q)\big)-h_D\big(\phi_2(Q)\big)=d\cdot h_C(Q)-d\cdot h_C(Q)+O(1)$, and we have the desired relation in (\ref{height}). 

Since $\phi_i(P)=n_i^{(\xi)}S+T_i$, we compute that $\hat{h}_D\big(\phi_1(P)\big)-\hat{h}_D\big(\phi_2(P)\big)=\big[{\big(n_1^{(\xi)}\big)}^2-{\big(n_2^{(\xi)}\big)}^2\big]\cdot\hat{h}_D(S)$. However, it is well known that $\hat{h}_D=h_D+O(1)$. (See, for example \cite[Theorem. B.5.1]{Silv-Hindry}). In particular, 
\begin{equation}{\label{canonicalheight}}\Big\vert{\big(n_1^{(\xi)}\big)}^2-{\big(n_2^{(\xi)}\big)}^2\Big\vert\leq\frac{M(\phi,\lambda,C,A)}{\hat{h}_D(S)},\end{equation} 
by equation (\ref{height}). Here the constant $M$ depends on $\phi,\lambda$, $C$, and $A$ but not on $\xi$. Moreover, since $A[G]\subseteq A[N]$ and $S$ is not torsion, it follows from \cite[Theorem 6]{SilvDuke} that 
\begin{equation}{\label{lower bound}}\hat{h}(S)=\hat{h}_{I_\xi^*D}(S_\xi)\geq c_1\big(d(\xi),g\big)\cdot\log\big\vert\Norm_{K/\mathbb{Q}}\Disc K_C(\xi)/K\big\vert-c_2\big(d(\xi),g,A,D\big),\end{equation} where the constant $c_1>0$ depends on $d(\xi)=|K_C(\xi)/K|$ and $g=\dim(A)$, and $c_2$ depends on $g$, $d(\xi)$, $A$ and our ample symmetric divisor $D$. However, the degree $d(\xi)$ in fact only depends upon $G$; see \cite[Remark 1.]{Silv}. Hence, we may choose $c_1$, $c_2$ and $M$ all independent of $\xi$. Therefore, after combining (\ref{canonicalheight}) and (\ref{lower bound}), we deduce that $n_1^{(\xi)}=\pm{n_2^{(\xi)}}$ for all but finitely many $\xi\in H^1(K,G)$. In particular, either 
\begin{equation}{\label{Torsion}}
\phi_1^{(\xi)}(P_{\xi})+\phi_2^{(\xi)}(P_{\xi})=T_1^{(\xi)}+T_2^{(\xi)}\;\;\; \text{or}\;\;\; \phi_1^{(\xi)}(P_{\xi})-\phi_2^{(\xi)}(P_{\xi})=T_1^{(\xi)}-T_2^{(\xi)}. \end{equation} This completes the proof of Lemma \ref{Lemma1}.    
\end{proof}
We now return to the proof of Theorem \ref{thm:General}. Let $\mathcal{U}$ be the set of counterexamples to Theorem \ref{thm:General}, that is \[\mathcal{U}=\Big\{\xi\in H^1(K,G)\Big\vert\;\rank\big(A_\xi(K)\big)\leq1,\;\text{and there exists}\; P_\xi\in C_\xi(K)\setminus C_\xi[G]\Big\}.\] It is our aim to show that $\mathcal{U}$ is a finite set. For a contradiction, suppose that we can choose an infinite subset $\{\xi_1,\xi_2,\dots\}\subseteq\mathcal{U}$. In particular, we can assume that the sequence \[\Big\{\log\big\vert\Norm_{K/\mathbb{Q}}\Disc K_C(\xi_i)/K\big\vert\Big\}\rightarrow\infty.\] Let $P_{\xi_i}\in C_{\xi_i}(K)$ be the points guaranteed by the $\xi_i$'s membership in $\mathcal{U}$, and let $P_i=i_{\xi_i}(P_{\xi_i})$ be their images in $C(\bar{K})$. Note that since $P_i\notin C[G]$, guaranteed by our assumption on $P_{\xi_i}$, we have a lower bound \[h_C(P_i)\geq c_3\log\big\vert\Norm_{K/\mathbb{Q}}\Disc K_C(\xi_i)/K\big\vert-c_4;\] see \cite[Theorem 3]{SilvDuke}. Moreover, $c_3$ and $c_4$ may be chosen independently of $\xi_i$ by \cite[Remark 1.]{Silv}. In particular, the subset $\{P_1,P_2,\dots\}\subseteq C(\bar{K})$ must also be infinite. 

Now from Lemma \ref{Lemma1}, after possibly discarding finitely many of the $\xi_i$'s, we see that either \[\phi_1^{(\xi_i)}(P_{\xi_i})+\phi_2^{(\xi_i)}(P_{\xi_i})\in A_{\xi_i}(K)_{\Tor}\;\; \text{or} \;\;\phi_1^{(\xi_i)}(P_{\xi_i})-\phi_2^{(\xi_i)}(P_{\xi_i})\in A_{\xi_i}(K)_{\Tor}.\] However, for all but finitely many $\xi_i$, if $T_{\xi_i}\in {(A_{\xi_i})}_{\Tor}$, then $T_{\xi_i}\in A_{\xi_i}[G]$. (Again, see \cite[Theorem 6]{SilvDuke}). But by assumption $1$, this means that $T_{\xi_i}$ is an $N$-torsion point. After applying $I_{\xi_i}$ to the equations in (\ref{lemma equation}), we get that \[\phi_1(P_i)+\phi_2(P_i)\in A[N]\;\;\; or \;\;\;\phi_1(P_i)-\phi_2(P_i)\in A[N].\] 
Since $A[N]\cong(\mathbb{Z}/N\mathbb{Z})^{2g}$ is a finite set, the pigeonhole principle implies that for some $T\in A[n]$, there are infinitely many $i$ such that either $\phi_1(P_i)+\phi_2(P_i)=T$ or $\phi_1(P_i)-\phi_2(P_i)=T$. However, since $C$ is a curve, we see that either $\phi_1+\phi_2$ or $\phi_1-\phi_2$ is a constant function. This contradicts assumption $4$, and we deduce that $\mathcal{U}$ is finite.            
\end{proof}
\begin{remark} One can easily apply Theorem \ref{thm:General} to the symmetric quartics defined on (\ref{Quartic}), or to the Fermat sextics $x^6+y^6=c$, which were also studied by Dem'janenko in \cite{Dem}. In the quartic case, one checks that the maps $\phi_1\pm\phi_2$ take the points at infinity on $F$ to the point at infinity $\mathcal{O}_E$ on $E$. Hence, if we take $\mathcal{O}_E$ to be the identity on $E$ and if one of the maps $\phi_1\pm\phi_2$ is constant, then $\phi_1=\pm{\phi_2}$. In either case, this equality forces $x=\pm{y}$, which is false for all but finitely many points on $F$. 

Alternatively, using the addition law on $E=E_{(a,b)}$, we see that $(x/z)\circ(\phi_1+\phi_2):F\rightarrow\mathbb{P}^1$ is given by \[(x/z)\circ(\phi_1+\phi_2)([x,y,z])=\frac{(2xy)^2+(2x+2y)^2(x^2+y^2)+4az^2(x^2+xy+y^2)+a^2z^4}{(x+y)^2z^2}.\] 
Since the map $x/z: E\rightarrow\mathbb{P}^1$ has degree $2$, it follows that \[deg(\phi_1)=deg(\phi_2)=4\;\;\;\text{and}\;\;\;\; deg(\phi_1+\phi_2)=8.\] \
If we let $D=(\mathcal{O}_E)\in Div(E)$, then we define a pairing $\langle\;,\; \rangle_{\text{\tiny D}}$ on $\Hom(F_{(a,b)},E_{(a,b)})$ given by 
\begin{equation}{\label{degree pairing}}\langle\;\phi,\psi \rangle_{\text{\tiny D}}=\frac{1}{2}\Big(deg(\phi+\psi)^*D-deg\phi^*D-deg\psi^*D\Big).\end{equation}
 We calculate that \[\langle\phi_1,\phi_1\rangle_{\text{\tiny D}}=4=\langle\phi_2,\phi_2\rangle_{\text{\tiny D}}\;\;\;\;\text{and}\;\;\;\; \langle\phi_1,\phi_2\rangle_{\text{\tiny D}}=\frac{1}{2}(8-4-4)=0\] It is well known that this pairing is nondegenerate, from which it follows that $\phi_1$ and $\phi_2$ are independent elements of $\Hom(F_{(a,b)},E_{(a,b)})$. From here, the conclusion of Theorem \ref{thm:General} can be achieved by \cite[Theorem 1]{Silv} in this case.      
\end{remark} 
Of course, the group law and rank computations are more difficult for abelian varieties of dimension greater than $1$. However, if $A$ is the Jacobian of some hyperelliptic curve, then bounding the rank of $A$, at least in principle, is possible \cite{descent}. With this in mind, we have the following corollary, which generalizes the work of Dem'janenko for the Fermat quartics. As a reminder, for odd integers $m\geq3$, we define 
\begin{equation}{\label{symmetric}} C=C_{(a,b)}: X^{2m}+aX^m+aY^m+Y^{2m}=b,\;\;\;\text{and}\;\ B=B_{(a,b)}:y^2=-x^{2m}-ax^{m}+(a^2/4+b).\end{equation}
Similarly, set $C_\alpha=C_{(\alpha\cdot a,\alpha^2\cdot b)}$ and $B_\alpha=B_{(\alpha\cdot a,\alpha^2\cdot b)}$ for all $\alpha\in K^*/(K^*)^m$. We identify the Jacobian $J$ of $B$ with $\Pic^0(B)$, and use $[\;\cdot\;]$ to denote a divisor class modulo linear equivalence. Note that we have maps $\phi_1,\phi_2:C\rightarrow J$, given by \[\phi_1(x,y)=\big[2\cdot\big(x\;,\;y^m+a/2\big)-(\infty_++\infty_-)\big],\;\;\text{and}\;\;\;\phi_2(x,y)=\big[2\cdot\big(y\;,\;x^m+a/2\big)-(\infty_++\infty_-)\big].\] 

We will show that $C$ and $J$ satisfy the hypotheses of Theorem \ref{thm:General}. In particular, let $G=\mu_m$ be the $m$-th roots of unity in $\bar{K}$, and note that $G$ acts on both $C$ and $J$. Specifically, for $\zeta\in G$ we have the action $\zeta\cdot(x,y)=(\zeta\cdot x,\zeta\cdot y)$ on $C$ and the action $\zeta\cdot(x,y)=(\zeta\cdot x,y)$ on $B$. Moreover, the action of $G$ on $B$ induces an action on $J$. Similarly, let $\lambda:\mathbb{P}^2\rightarrow\mathbb{P}^2$ be the map $[x,y,z]\rightarrow [y,x,z]$, and note that $\lambda(C)\subseteq C$ and that $\phi_2=(\phi_1\circ\lambda|_{C})$ on $C$. 

Finally, from Kummer theory, we have that \[H^1(K,G)=H^1\big(\Gal(\bar{K}/K),\mu_m\big)\cong K^*/(K^*)^m,\] and if $\xi\in H^1(K,G)$ corresponds to $\alpha\in K^*/(K^*)^m$, then the twists of $C$ and $J$ are precisely $C_\alpha$ and $J_\alpha=\Jac(B_\alpha)$. With this in place, we restate and prove the following result.      
\Symm*
\begin{proof} We need only verify the conditions of Theorem \ref{thm:General}. Conditions $2$ and $3$ are straightforward, and one checks that $J[G]\subseteq J[m]$ for condition $1$. To verify condition $4$, suppose that $\phi_1+\phi_2$ is a constant map from $C\rightarrow A$. Then by evaluating this map on the points at infinity of $C$, we see that $\phi_1+\phi_2=\mathcal{O}_J$, the trivial divisor class. Hence, $\phi_1(P)=-\phi_2(P)$ for all $P\in C(\bar{K})$. We multiply this equation by $(m-1)/2$. In particular, for all $P=(x,y)\in C(\bar{K})$, we have that \begin{equation}{\label{div}}\Big[(m-1)\cdot\big(x,y^m+a/2\big)+\Big(\frac{m-1}{2}\Big)\cdot(\infty_++\infty_-)\Big]=\Big[(m-1)\cdot\big(y,-(x^m+a/2)\big)+\Big(\frac{m-1}{2}\Big)\cdot(\infty_++\infty_-)\Big].\;\;\end{equation} \indent It follows from \cite[Prop.1]{hyperelliptic} that if $(y^m+a/2)\cdot(x^m+a/2)\neq0$, then $(x,y^m+a/2)=\big(y,-(x^m+a/2)\big)$. This conclusion holds because the divisors on (\ref{div}) are effective and reduced, (effective since the genus of $B$ is $m-1$ and reduced since we miss the Weierstrass points). Hence, if $\phi_1+\phi_2$ is a constant function, then $x=y$ (at least when $(y^m+a/2)\cdot(x^m+a/2)\neq0$). But $x\neq y$ for all but finitely many points on $C$. We conclude that $(\phi_1+\phi_2):C\rightarrow J$ is indeed non-constant.

On the other hand, suppose that $\phi_1-\phi_2$ is a constant map. As before, we apply our map to the points at infinity on $C$. Half of these points map to $2\cdot[\infty_+-\infty_-]$ and the other half map to $2\cdot[\infty_--\infty_+]$. Since we are assuming that $\phi_1-\phi_2$ is constant, it must be the case that $4\cdot[\infty_+-\infty_-]=0$. However, the divisor $[\infty_+-\infty_-]$ is already an $m$-torsion point. In fact, the divisor of the function $f=\Big(y - \sqrt{-1}\cdot\big(x^m + -a/2z^m\big)\Big)/z^m$ gives the desired relation. Since $m$ is odd, it follows that $[\infty_+-\infty_-]=\mathcal{O}_J$ is the trivial class. Therefore, $\phi_1(P)-\phi_2(P)=2\cdot[\infty_+-\infty_-]=\mathcal{O}_J$ for all $P\in C(\bar{K})$. We now repeat the argument (without the minus signs) as in the $\phi_1+\phi_2=\mathcal{O}_J$ case to show that this is impossible for all $P$. In particular, we have verified the hypotheses of Theorem \ref{thm:General}.  

The description of the possible $K$-points follows from Theorem \ref{thm:General} and the fact that any point on $C_\alpha$ which is fixed by some nontrivial root of unity must be at infinity ($z$-coordinate zero in $\mathbb{P}^2$).    
\end{proof}
\begin{remark}We illustrate the difficulty one has in computing the degree pairing defined on (\ref{degree pairing}) when $m=3$. Recall that $J$ is the Jacobian of the genus two curve $B:y^2=-x^6-ax^3+(a^2/4+b)$. It is known that $J$ is isogenous to $E\times E$, for some elliptic curve $E$ defined over an algebraic closure; see \cite[Prop 4.2]{genus2}. Hence, one can try pushing our maps forward to $E$, where the group law is explicit, and mimic the degree calculation as in the quartic case. Specifically, if $c=a^2/4+b\neq0$ and $d=a\cdot i/\sqrt{c}$, then we have the elliptic curve \[E: Y^2=X^3-(3d-30)X^2+(3d^2+36d+60)X-(d-2)(d+2)^2,\] and a map \[\pi:C\rightarrow E,\;\;\;\text{given by}\;\;\; \pi(x,y)=\left(\;\frac{\left(\frac{x}{\sqrt[6]{c}\cdot i}+1\right)^2}{\left(\frac{x}{\sqrt[6]{c}\cdot i}-1\right)^2}\cdot(d+2)\;,\;\frac{8y\cdot(d+2)}{\sqrt{c}\cdot\left(\frac{x}{\sqrt[6]{c}\cdot i}-1\right)^3}\;\right).\] However, for studying dependence, there is no harm in passing to an extension field. Let $\tilde{\pi}_i:C\rightarrow E$ be the map factoring through $\phi_i$ and the isogeny $J\sim E\times E$. One checks that $2\pi_i+T=\tilde{\pi_i}$ for some $3$-torsion point $T$ on $E$. Hence, it suffices to compute the degrees of $\pi_1, \pi_2$, and $\pi_1+\pi_2$ to show the independence of $\phi_1$ and $\phi_2$. However, even if one chooses $a$ and $b$ such that $E, \pi_1$ and $\pi_2$ are defined over $\mathbb{Q}$, Magma could not compute $\deg(\pi_1+\pi_2)$ given $12$ hours with a personal computer.      
\end{remark} 
We further illustrate Theorem \ref{thm:General} with an interesting family of twists in which every curve corresponding to a single positive, odd prime has positive rank (assuming the standard parity conjecture). Hence, using Dem'janenko's method seems indispensable if one wishes to catalogue the rational points in this particular family. Before we restate Theorem \ref{thm:Hasse}, we remind the reader  of the parity conjecture. For a more comprehensive presentation, as well as the relevant definitions, see \cite[\S1]{Parity}. 
\begin{conjecture}{\textbf{(Parity Conjecture)}}{\label{parity conjecture}} \; Let $E$ be an elliptic curve defined over the rational numbers, let $L(E/\mathbb{Q},s)$ be the $L$-function of $E$, and let $W(E/\mathbb{Q})=\pm{1}$ be the global root number of $E$. Then \[(-1)^{\rank\left(E(\mathbb{Q})\right)}=W(E/\mathbb{Q}).\] 
\end{conjecture} 
Although the parity conjecture remains an open problem, we can use it, in combination with Theorem \ref{thm:General}, to predict the existence of many curves of genus $3$ which break the Hasse local-global principle. 
\Hasse*
\begin{proof} Note that $p>0$ implies that $F^{(p)}(\mathbb{R})$ is nonempty, since $F^{(p)}\cong F$ over $\mathbb{R}$, and $(1,1)$ is a point in $F(\mathbb{R})$. For the rest of the proof of statement $1$, we first show that $F^{(p)}$ has $\mathbb{Q}_q$-points for $q=2,3,p$ (the primes of bad reduction). Suppose that $p\equiv1\bmod{24}$, so that in particular $p\equiv1\bmod{8}$ and $p\equiv1\bmod{3}$. Hence, $p=\theta^2$ for some $\theta\in\mathbb{Z}_2$. One checks that $(\theta,\theta)\in F^{(p)}(\mathbb{Q}_2)$. Similarly, we can write $p=\omega^2$ for some $\omega\in\mathbb{Z}_3$. It follows that $(\omega,\omega)\in F^{(p)}(\mathbb{Q}_3)$. On the other hand, $p\equiv1\bmod{8}$ implies that we can choose a primitive $8$-th root of unity $\zeta$ in $\mathbb{Q}_p$. Hence, $[\zeta,1,0]\in F^{(p)}(\mathbb{Q}_p)$ is a rational point at infinity. 

Now for $q\notin\{2,3,p\}$, a simple argument with Hensel's lemma shows that $F^{(p)}(\mathbb{Q}_q)\neq\varnothing$ if and only if $F^{(p)}(\mathbb{F}_q)\neq\varnothing$. However, since $F^{(p)}$ has genus $3$, it follows from the Weil bound that $F^{(p)}(\mathbb{F}_q)\neq\varnothing$ for all $q\geq37$. For the remaining small primes, we check manually (with Magma \cite{Magma}) that $F^{(p)}(\mathbb{F}_q)$ has points for all possible congruence classes of $p$.

For part $2$, if $p>0$ is an odd prime, then it suffices to compute $W_2$ and $W_p$, the local root numbers at $2$ and $p$ respectively. Since $E$ has complex multiplication, $E^{(p)}$ has potential good reduction at every prime. Moreover, Tate's algorithm shows that $E^{(p)}$ has Kodaira type $I_0^*$ at $p$, and \cite[Theorem 1.1]{rootnumbers} implies that $W_p=\left(\frac{-1}{p}\right)$. On the other hand, one checks that $E^{(p)}$ has Kodaira type $III$ at $2$, that $c_4(E_p)=2^5\cdot5\cdotp^2$ and that $c_6(E_p)=2^8\cdot7\cdot p^3$. It follows from Halberstadt's table that $W_2=1$ if and only if $6p+5\equiv\pm{1}\bmod8$. We conclude that $W_2=-1$ just in case $p\equiv3\bmod{4}$. Hence, the global root number $W\left(E^{(p)}\right)=W_2\cdot W_p=-1$ as claimed.       

As for the proof of statement $3$, we will utilize the fact that the maps $\phi_1,\phi_2:F\rightarrow E$ are of the same degree, adapting the proof of Theorem \ref{thm:General}, in order to make the bounds explicit. This method is also discussed in \cite[\S13.3.1]{Coh}. 

Suppose that $P_{\alpha}=(x_0,y_0)\in F^{(\alpha)}(\mathbb{Q})$ and that $\rank\left(E^{(\alpha)}(\mathbb{Q})\right)\leq1$. Therefore, we can write $\phi_i(P_{\alpha})=n_i\cdot G_{\alpha}+T_i$, where $G_{\alpha}$ generates the free part of $E^{(\alpha)}(\mathbb{Q})$, if it exists, and some torsion points $T_i$. Furthermore, let $G$ be the image of $G_{\alpha}$ in $E$ under the map $(x,y)\rightarrow\left(\frac{x}{\alpha},\frac{y}{\sqrt{\alpha}^3}\right)$. 

Since $E$ has CM by $\mathbb{Z}[\sqrt{-2}]$, it follows that $E^{(\alpha)}$ is supersingular at all but finitely many primes $p\equiv5,7\bmod{8}$ (the inert primes). In particular, \[\big\vert E^{(\alpha)}(\mathbb{Q})_{\text{Tor}}\big|\;\;\text{divides}\; \gcd\Big(\big\{p+1\;\big\vert\; \text{almost all primes}\; p\equiv5,7\bmod{8}\big\}\Big).\] Dirchlet's theorem on arithmetic progressions implies that the above greatest common divisor is $2$. (See \cite[Theorem 1.1]{me} for details.) In particular, each torsion point $T_i$ is either the point at infinity or $(0,0)$, and so there is no need to distinguish between the rational torsion points of $E$ and $E^{(\alpha)}$.  

Now, we see that $h(\phi_i(Q))=2h(Q)+O(1)$ for all $Q\in F(\bar{\mathbb{Q}})$. (Again, see \cite[Theorem B.2.5]{Silv-Hindry}, where $F$ inherits the Weil height from its embedding in projective space.) In fact, we can make this bound explicit since the Nullstellensatz step of the proof of the lower bound is easy for our curve; see Proposition \ref{height prop} in the Appendix \ref{Appendix}. Hence, it follows from the defining equation of $F$ that: \[\Big|h\big(\phi_1(Q)\big)-h\big(\phi_2(Q)\big)\Big|\leq 8\log(2)+3\log(3)\leq8.842.\] On the other hand, we use Sage \cite{Sage} to compute a bound between the canonical and Weil heights on $E$, namely: \[-6.214\leq\hat{h}(W)-h(W)\leq 5.659\;\;\; \text{for all}\;\;\; W\in E(\bar{\mathbb{Q}}).\] Finally, for all square free twists $\alpha$, we can compute a lower bound (depending on $\alpha$) on the canonical height of any non-torsion point of $E^{(\alpha)}(\mathbb{Q})$. This has been made explicit in \cite[Theorem 1]{twistbound}, and we conclude that 
\begin{equation}{\label{twisted height}}\hat{h}(W)\geq\frac{1}{8}\log|\alpha|-0.8232, \;\;\;\text{for all}\;\;\; W\in E^{(\alpha)}\left(\mathbb{Q}\right)\setminus E^{(\alpha)}[2]. \end{equation}      
We are now ready to apply these bounds to the theorem. Note that $P=\left(\frac{x_0}{\sqrt{\alpha}}, \frac{y_0}{\sqrt{\alpha}}\right)\in F(\bar{\mathbb{Q}})$, and because $\phi_1$ and $\phi_2$ commute with the action of $\mathbb{Z}/2\mathbb{Z}$, we see that $\phi_i(P)=n_i\cdot G+T_i$. In particular, we may apply the above bounds to the case where $Q=P$ and $W=\phi_i(P)$. We get that \[|n_1^2-n_2^2|\leq\frac{5.659+8.842+6.214}{\hat{h}(G)};\] see \cite[\S13.3.1]{Coh}. However, if $\alpha>7\cdot10^{74}$, then $\hat{h}(G)>20.715$ by (\ref{twisted height}). This forces $|n_1^2-n_2^2|=0$, and hence $n_1=\pm{n_2}$. We can conclude that 
\begin{equation}{\label{dependence relation}} \phi_1(P_{\alpha})\pm\phi_2(P_{\alpha})\in E^{(\alpha)}\left(\mathbb{Q}\right)_{\text{Tor}}\subset E^{(\alpha)}[2].\end{equation} 
 In particular, $\phi_1(P_{\alpha})\pm\phi_2(P_{\alpha})\in\{\mathcal{O},(0,0)\}$. If $\phi_1(P_{\alpha})=\pm\phi_2(P_{\alpha})$, then $x=y$. This implies that $x^4-4\alpha x^2+3\alpha^2$ has a rational root. However, since $\alpha$ is square free, this forces $\alpha=3$. On the other hand, if $\phi_1(P_{\alpha})\pm\phi_2(P_{\alpha})=(0,0)$, then the group law on $E^{(\alpha)}$ implies that $-x^2y^2=-2\alpha^2$. Again this is impossible, since $\sqrt{2}$ is irrational. Hence, we deduce that $F^{(\alpha)}\left(\mathbb{Q}\right)=\varnothing$.  

As for the proof of statement $4$, we use the fact that $E$ has a rational $2$-torsion point and preform a $2$-descent. Suppose that $\rho:E^{(p)}\rightarrow A^{(p)}$ is the standard isogeny of degree $2$ with kernel $\{\mathcal{O},(0,0)\}$; see for instance \cite[10.4.9]{SilvElliptic}. Let \[C_{d}: dw^2=d^2-8pdz^2+8p^2z^4.\] Then the $\rho$-Selmer group consists of those $d\in\{\pm{1},\pm{2},\pm{p},\pm{2p}\}$ for which $C_d$ has points everywhere locally. Note that if $d<0$, then $C_d(\mathbb{R})=\varnothing$. Consider the case when $d=p$ and suppose that $(w,z)\in C_p(\mathbb{Q}_p)$. We see that in this case $w^2=p-8pz^2+8pz^4$. Hence the valuation (denoted $v_p$) of the left hand side of the equation is even. On the other hand, the valuation on the right hand side is $1+v_p(1-8z^2+8z^4)$. Thus $v_p(1-8z^2+8z^4)$ must be odd, in particular non-zero. If $v_p(z)\neq0$, then $v_p(1-8z^2+8z^4)=0$. Hence, we have that $v_p(z)=0$. However, $2v_p(w)\geq\min\{v_p(p),v_p(8pz^2),v_p(8pz^4)\}=1$, so that $(w,z)\in C_p(\mathbb{Z}_p)$ has integral coordinates. By reducing $\bmod{(p)}$, we see that the polynomial $f(x)=x^4-4x^2+2$ has a root in $\mathbb{F}_p$. This implies that $p\equiv\pm{1}\bmod{16}$; see \cite{quartic}. So, if $p\notequiv\pm{1}\bmod{16}$, then the $\rho$-Selmer group has $2$-rank less than or equal to one. It follows that $\rank\left(E^{(p)}(\mathbb{Q})\right)\leq2$ for such $p$; see, for instance, [Equation (5)]\cite{maximalrank}. 

Piecing everything together, if $p\equiv25\bmod{48}$, then $F^{(p)}$ is everywhere locally solvable and $\rank\left(E^{(p)}(\mathbb{Q})\right)\leq2$ by parts $1$ and $4$ of Theorem \ref{thm:Hasse}, respectively. Moreover, the Parity Conjecture \ref{parity conjecture} and part $2$ of Theorem \ref{thm:Hasse} imply that the rank of $E^{(p)}$ is odd. Hence $\rank\left(E^{(p)}(\mathbb{Q})\right)=1$. Finally, if $p>7\cdot10^{74}$ and $\rank\left(E^{(p)}(\mathbb{Q})\right)=1$, then $F^{(p)}$ has no rational solutions by part $3$. In particular, Dirichlet's theorem on arithmetic progressions implies that there are infinitely many primes $p$ for which $F^{(p)}$ breaks the Hasse principle.                     
\end{proof}
Of course, in practice part $3$ of Theorem \ref{thm:Hasse} holds for much smaller integer values of $\alpha$. For instance, we checked manually all values $1\leq\alpha\leq500$, to see whether $\rank\left(E^{(\alpha)}(\mathbb{Q})\right)\leq1$ implies that $F^{(\alpha)}(\mathbb{Q})=\varnothing$. Outside of the small exception $\alpha=3$, the conclusion held. This begs the question as to whether the result holds for all square-free $\alpha\geq3$.

The first case in which we are not able to conclude anything about the set of rationals points is when $p=577$. One can show that $\rank\left(E^{(577)}(\mathbb{Q})\right)=3$. Hence, Dem'janenko's method does not apply. Morever $577\equiv1\bmod{24}$, so that $F^{(577)}$ has points everywhere locally. To make matters more difficult, the rank of the Jacobian of $F^{(577)}$ is at least $6$, from which it follows that Chabauty's method is not directly applicable. We ask whether $F^{(577)}(\mathbb{Q})$ is indeed empty, perhaps provable by some combination of covering collections with Chabauty.
\begin{remark}Similarly, when $(a,b)=\left(\frac{-21}{4},\frac{-889}{16}\right)$, then $F$ covers an elliptic curve $E$ with CM by an order in $\mathbb{Q}(\sqrt{-7})$. Recently Coates, Li, Tian, and Zhai determined the rank of many twists of $E$ and verified the BSD conjecture in these cases \cite{Coates}. Unfortunately, $F^{(\alpha)}(\mathbb{R})=\emptyset$ for all $\alpha$, and so there is no point in using Theorem \ref{thm:General} to determine its rational points. However, there may be some use in studying $F^{(\alpha)}(K)$ over quadratic extensions $K/\mathbb{Q}$, using the results in \cite{Coates} and the techniques discussed in this note. We do not take up that problem here.          
\end{remark} 
We now discuss some motivation coming from orbit intersection problems in quadratic dynamics.    
\begin{section}{3. Shifted Orbit Equations}{\label{Orbit}}

\indent \indent An interesting problem in dynamics is to understand how orbits intersect after applying a linear (or more complicated) shift. To make this precise, we fix some notation. Let $K$ be a field and let $a,b,\alpha$ and $\beta$ be elements of $K$. For a rational map $f\in K(x)$, we define \[\mathcal{O}_{f,n}(\alpha)=\{f^{n}(\alpha),f^{n+1}(\alpha),\dots\}.\]Stated more precisely, our problem is to understand the set $L\big(\mathcal{O}_{f,n}(\alpha)\big)\cap\mathcal{O}_{f,n}(\beta)$. In particular, we would like to study the size of this intersection. Note that this is equivalent to studying the rational solutions to an equation of the form $f^{n}(y)-af^n(x)=b$. Curves defined in this way, along with encoding dynamical information, can give rise to some interesting arithmetic. 

For example, suppose that $K=\mathbb{Q}$, $f(x)=x^2-2, L(x)=1-x$ and $n=2$. Note that $f$ has the five rational preperiodic points $\{0,\pm{1},\pm{2}\}$, often written $\PrePer(f,\mathbb{Q})=\{0,\pm{1},\pm{2}\}$, and that \[\mathcal{O}_{f,2}(0)=\mathcal{O}_{f,2}(-2)=\mathcal{O}_{f,2}(2)=\{2\}=L\big(\mathcal{O}_{f,2}(-1)\big)=L\big(\mathcal{O}_{f,2}(1)\big).\] Stated geometrically, for $X_{2^m}:=\big\{(x,y)\big\vert\; f^m(x)+f^m(y)=1\big\}$, we deduce the following curious fact:
\begin{equation}{\label{ChebychevPoints}} \{(0,\pm{1}),(\pm{2},\pm{1}),(\pm{1},0),(\pm{1},\pm{2})\}\subseteq X_{2^m}(\mathbb{Q}). 
\end{equation} 
\indent In particular, we have constructed curves of arbitrarily large genus having at least $12$ rational points. Note that since $X_{2^n}\rightarrow X_{4}$, via $(x,y)\rightarrow \big(f^{n-2}(x),f^{n-2}(y)\big)$, and that $X_4(\mathbb{Q})$ is finite (genus $3$), our dynamical problem of bounding the size of the intersection of these shifted orbits amounts to computing $X_4(\mathbb{Q})$. Namely, is our list on (\ref{ChebychevPoints}) exhaustive? 

However, after rearranging terms, $X_4$ becomes $F_{(-4,3)}:x^4-4x^2+y^4-4y^2=-3$, and one computes that $E_{(-4,3)}$ has rank one. Hence, we may use Dem'janenko's method to describe $X_4(\mathbb{Q})$ explicitly. 

\begin{remark}Note that the symmetric quartics $F_{(a,b)}$ will appear in this dynamical problem (after a change of variables) whenever $f$ is quadratic polynomial and $L$ is of the form $L(x)=b-x$. For a general linear transformation $L$, the equation $f^2(y)=L\big(f^2(x)\big)$ will map to two (distinct) elliptic curves, and so Dem'janenko's method will not necessarily apply. 
\end{remark}  
In this example, $f=x^2-2$ is the Chebychev polynomial of degree $2$, and by generalizing $X_{2^n}$ to incorporate all Chebychev polynomials, we obtain new curves with similar properties. Here we consider the \emph{Chebyshev polynomial} $T_d$ as characterized by the equation \[T_d(z+z^{-1})=z^d+z^{-d}\;\;\; \text{for all}\;\; z\in\mathbb{C}^*.\] Furthermore, $T_d$ is known to be a degree $d$ monic polynomial with integer coefficients. The classical Chebyshev polynomials $\tilde{T}_d$ were defined in the following way: \[\text{If we write}\;\; z=e^{it},\;\;\text{then}\;\;\; \tilde{T}_d(2\cos(t))=2\cos(dt),\]  though we will use the first characterization where $T_d$ is monic. For a more complete discussion of these polynomials, see \cite[\S6.2]{SilvDyn}. 

With these definitions in place, we define the \emph{Chebychev Curve of degree d} to be
\begin{equation}{\label{ChebychevCurve}} X_d:=\big\{(x,y)\vert\; T_d(x)+T_d(y)=1\big\}. 
\end{equation} 

It is shown in Proposition \ref{nonsingularity} of the Appendix that $X_d$ is nonsingular. Note that $X_n\rightarrow X_{d}$ whenever $d|n$, as is the case for the Fermat equation $x^n+y^n=1$. This follows from the well known fact that $T_{n\cdot m}(x)=T_n\big(T_m(x)\big)$, called the \emph{nesting property}. (Again, see \cite[\S6.2]{SilvDyn}). 

\begin{remark}We will only consider rational points on the affine models defined in (\ref{ChebychevCurve}) above. The reason is that the rational points at infinity on the projective closure of $X_d$ in $\mathbb{P}^2$ are easy to describe: If $d$ is odd, then $[1,-1,0]$ is the only rational point at infinity, and if $d$ is even, then there are no points. 
\end{remark} 
As we will see, $X_d(\mathbb{Q})$ has several rational points for most values of $d$, making it an interesting family in that sense. As a first step, we have the following result.  
\Chebychev*
To prove Theorem \ref{thm:Chebychev}, we are in need of the following elementary lemma: 
\begin{lemma}{\label{Chebychev Lemma}} Suppose that $d\notequiv0\bmod{3}$. Then $T_d$ takes the following values: 
\begin{enumerate}
\item If $d$ is odd, then $T_d$ fixes every element of $\{0,\pm{1},\pm{2}\}$. 
\item If $d\equiv2\bmod{4}$, then $T_d(0)=-2, T_d(\pm{2})=2$, and $T_d(\pm{1})=-1$. 
\item If $d\equiv0\bmod{4}$, then $T_d(0)=T_d(\pm{2})=2$ and $T_d(\pm{1})=-1$.    
\end{enumerate} 
\end{lemma}
\begin{proof} Since $2=1+1^{-1}$, we see that $T_d(2)=1^d+1^{-d}=2$ for all $d\geq1$. Similarly, if $\zeta$ is a primitive third root of unity, then $-1=\zeta+\zeta^{-1}$. Since $d\notequiv0\bmod{3}$, we see that $T_d(-1)=\zeta^d+\zeta^{-d}=\zeta+\zeta^{-1}=-1$. Finally, note that $0=i-i=i+i^{-1}$, so that $T_d(0)=i^d+i^{-d}$.

We can now prove the descriptions in Lemma \ref{Chebychev Lemma}. If $d$ is odd, then $T_d$ is an odd function \cite[Prop. 6.6]{SilvDyn}. Hence $T_d(-2)=-T_d(2)=-2$ and $T_d(1)=-T_d(-1)=-(-1)=1$. Moreover, $T_d(0)=i^d+(-i)^d=i^d(1-1)=0.$ 

Similarly, if $d$ is even, then $T_d(-2)=T_d(2)=2$ and $T_d(1)=T_d(-1)=-1$, since $T_d$ is an even function \cite{SilvDyn}. On the other hand, $T_d(0)=i^d+(-i)^d=-1+(-1)=-2$ if $d\equiv2\bmod{4}$, and $T_d(0)=i^d+(-i)^d=1+1=2$ if $d\equiv0\bmod{4}$.           
\end{proof}
\begin{proof} We return to the proof of Theorem \ref{thm:Chebychev}. 
For $d=3$, we compute with Sage that the Jacobian of $X_3$ is the elliptic curve $B: y^2=x^3-27x+189/4$. Moreover, a $2$-descent with Sage shows that $B(\mathbb{Q})$ has rank zero and is trivial \cite{Sage}. It follows that $X_3(\mathbb{Q})=\varnothing$. More generally, if $3|d$, then write $d=3\cdot d'$. Hence, $X_d\rightarrow X_3$ via $(x,y)\rightarrow\big(T_{d'}(x),T_{d'}(y)\big)$, and we conclude that $X_d(\mathbb{Q})=\varnothing$. 

If $d\equiv0\bmod{4}$ and $d=4\cdot d'$, then $X_d\rightarrow X_4$ via $(x,y)\rightarrow \big(T_{d'}(x),T_{d'}(y)\big)$. Therefore, it will suffice to compute $X_{4}(\mathbb{Q})$, to compute $X_d(\mathbb{Q})$. 

After rearranging terms, $X_4$ becomes $x^4-4x^2-4y^2+y^4=-3$, which takes the form of the symmetric quartics we have studied. In particular, $X_4$ has two independent maps $\phi_1,\phi_2:X_4\rightarrow Y$ where $Y:y^2=x\cdot(x^2+16x-16)$. Note that $Y=E_{(-4,-3)}$ in our earlier notation. We compute with Sage that $Y(\mathbb{Q})$ has rank one with generator $G=\big(4,-16\big)$ and rational torsion subgroup $\mathbb{Z}/2\mathbb{Z}$; see \cite{Sage}.  
 
 As in Theorem \ref{thm:Hasse} and Theorem \ref{thm:General}, if $P\in X_4\big(\mathbb{Q}\big)$, then we may write $\phi_i(P)=n_i\cdot G+T_i$ and compute the bounds \[ |n_1^2-n_2^2|\leq\frac{(\log(24)+\log(192))+2\cdot9.62}{\hat{h}(G)}\leq78;\] see Proposition \ref{height prop} in the Appendix. Hence, if $n_1\neq n_2$, then $\max\{|n_1|,|n_2|\}\leq40$; see \cite[\S13.3.1]{Coh}. Now, we simply check whether the points $n\cdot G$ and $n\cdot G+(0,0)$ have rational preimages under $\phi_1$ for all $|n|\leq40$. However, the points $n\cdot G$ have positive $x$-coordinate and so cannot have rational preimages under $\phi_1$; see equation (\ref{maps}). On the other hand, the points $G+T$ and $2\cdot G+T$ give the rational points corresponding to $x=\pm{1}$ and $x=\pm{2}$, and we check with Sage that these are the only combinations of $n\cdot G+T$ which do so. 
 
 The remaining rational points correspond to the case when $n_1=n_2$ and  $\phi_1\pm\phi_2=\mathcal{O},(0,0)$. One easily checks that in this case $x\cdot y=0$, which completes the description of $X_4\big(\mathbb{Q}\big)$. 
 
In general, if $(x_0,y_0)\in X_d(\mathbb{Q})$ and $d=4\cdot d'$, then $T_{d'}(x_0)\in\{0,\pm{1},\pm{2}\}$ by our description of $X_4(\mathbb{Q})$. Since, $T_{d'}\in\mathbb{Z}[x]$ is monic, we see that $x_0\in\mathbb{Z}$ (it is a rational number which is integral over $\mathbb{Z}$). On the other hand, it is known that 
 \begin{equation}{\label{Chebychev Inequality}} |T_n(z)|\geq|T_{n-1}(z)|\geq\dots\geq |T_2(z)|\geq0,\;\;\text{for all}\;|z|\geq2\;\text{and}\;n\geq2.
 \end{equation}  
 here $z$ is a complex number; see \cite{Chebychev} for details. In particular, $|T_{d'}(x)|\geq |T_{2}(3)|=7$ for all real numbers $|x|\geq3$, since $|T_2(x)|=|x^2-2|$ is an increasing function on $(\sqrt{2},\infty)$ and decreasing on $(-\infty,-\sqrt{2})$. Hence, $x_0\in\{0,\pm{1},\pm{2}\}$, and our description of $X_d(\mathbb{Q})$ follows from part $3$ of Lemma \ref{Chebychev Lemma}.    

Similarly, we determine $X_5(\mathbb{Q})$ to describe $X_d(\mathbb{Q})$ for $d\equiv0\bmod{5}$. After expanding $T_5(x)+T_5(y)$ in terms of the symmetric polynomials $u=x+y$ and $v=x^2+y^2$, we map $X_5$ to the genus two curve \[C: -1/4u^5 + 5/2u^3 + 5/4uv^2 - 15/2uv + 5u-1=0,\] which is birational to the hyperelliptic curve \[C': y^2=5x^6-50x^4+125x^2+20x,\;\; \text{via}\;\; (u,v)\rightarrow(u,5uv-15u).\] One computes with Magma that the Jacobian of $C'$ has rank $1$; see \cite{Magma}. Moreover, since $C'$ has genus $2$, we can use the method of Chabauty and Coleman \cite{Poonen} to show that $C'(\mathbb{Q})=\{(0,0),(1,\pm{10})\}$. In fact, this process has been implemented in Magma. After computing preimages, we determine that $X_5(\mathbb{Q})=\{(0,1),(1,0),(-1,2),(2,-1)\}$ as claimed.

In particular, if $(x_0,y_0)\in X_d(\mathbb{Q})$ and $d=5\cdot d'$, then $T_{d'}(x_0)\in\{0,\pm{1},2\}$ by our description of $X_5(\mathbb{Q})$. As in the previous case, $x_0$ must be an integer and the inequality (\ref{Chebychev Inequality}) forces $x_0\in\{0,\pm{1},\pm{2}\}$. Form here, we determine $X_d(\mathbb{Q})$ by reading off the possible images of $T_d(x_0)$ using parts $1$ and $2$ of Lemma \ref{Chebychev Lemma} (we do not include the case when $d$ is divisible by $4$, since that was already completed).            
\end{proof}  
Given Theorem \ref{thm:Chebychev} (and evidence in the $d=7$ and $d=11$ cases), we restate the following conjecture.\\
\\
\; \textbf{Conjecture 1.}\;$If$\;$P=(x,y)\in X_d(\mathbb{Q})\; and\; d\geq3,\: then\; x,y\in\{0,\pm{1},\pm{2}\}.\; In\; particular, \#X_d(\mathbb{Q})=0,4,8\; \text{or} \;12, and\; X_d(\mathbb{Q})\; can\; be\; described\; explicitly\; in\; terms\; of\; d\bmod{12}.$  
\\
\\
Note that it suffices to establish the conjecture for primes, analogous to the Fermat curves. This follows from the nesting property and the fact that if $x_0\in\mathbb{Z}$ and $|T_d(x_0)|\leq2$, then $|x_0|\leq2$; see the inequality on (\ref{Chebychev Inequality}) above. 
\end{section} 
\\
\\ 
\indent \textbf{Acknowledgments} I thank Joe Silverman for the many discussions  related to the topics appearing in this note. I also thank Vivian Olsiewski Healey for her suggestions during the editing process.
\section{4. Appendix}{\label{Appendix}} 
\indent\indent We collect and prove some elementary facts used throughout this paper. In keep with our notation, let $F=F_{(a,b)}$ and $E=E_{(a,b)}$ be the symmetric quartic and elliptic curves defined in Section \ref{Intro}. Furthermore, let $h_F$ be the Weil height on $F$ given by its embedding into $\mathbb{P}^2$, and let $h_E$ be the usual height function on $E$ given by its $x$-coordinate. (See \cite[VIII.6]{SilvElliptic}). Finally, for a number field $K$, let $M_K$ be the complete set of absolute values on $K$. With these definitions is place, we have the following height estimates.     
\begin{proposition}{\label{height prop}} Let $F=F_{(a,b)}$ and $E=E_{(a,b)}$ be the symmetric quartic and elliptic curves defined above. Then for all $P\in E_{(a,b)}(\bar{K})$, we have the bounds \begin{equation}{\label{height estimate}} 2h_F(P)-\log(12)-\log\big(\kappa_{(a,b)}\big)\leq h_E\big(\phi_i(P)\big)\leq 2h_F(P)+\log(24), \end{equation} where \[ \kappa_{(a,b)}=\prod_{v\in M_K}\max\Big\{\Big\vert\frac{1}{4}\Big\vert_v,\Big\vert\frac{a}{4}\Big\vert_v,\;|a|_v,\;|b|_v\Big\}.\]
\end{proposition} 
\begin{proof} We make the relevant objects from the proof of \cite[Theorem B.2.5]{Silv-Hindry} explicit, and then refer the reader there. In particular, the usually difficult lower bound on the height, coming from the Nullstellensatz, is relatively easy for our curves. 

We prove the bounds for $\phi_1$ only, as the proof for $\phi_2$ is identical. In particular, we are studying heights with respect to the morphism $(x/z)\circ\phi_1:F\rightarrow\mathbb{P}^1$. In keep with the notation from the proof of \cite[Theorem B.2.5]{Silv-Hindry}, let $f_0=-4x^2$, $f_1=z^2$, and $p_1=x^4+ax^2z^2+ay^2z^2+y^4-bz^4$. The upperbound is given by the usual 
\begin{equation} h_E\big(\phi_1(P)\big)\leq2h_F(P)+h(\phi_1)+\log\binom{4}{2}.
\end{equation} 
Now for the explicit Nullstellensatz step. We first write $x^4,y^4$ and $z^4$ as a combination of $f_0,f_1$, and $p_1$. Specifically,    
\begin{equation}{\label{expansion}}\begin{split}
 x^4=\;&(-1/4x^2)\cdot f_0+(0)\cdot f_1+(0)\cdot p_1, \\
y^4=\;&(1/4x^2+a/4z^2)\cdot f_0+(bz^2-ay^2)\cdot f_1+(1)\cdot p_1,\\
z^4=\;&(0)\cdot f_0+(z^2)\cdot f_1+(0)\cdot p_1.\end{split}   
\end{equation}  
Let $g_{ij}$ be the coefficient functions in (\ref{expansion}) above. Again, in keep with the notation from the proof of \cite[Theorem B.2.5]{Silv-Hindry}, it follows that \[ |P|_v\leq\epsilon_v(2)\cdot\epsilon_v(6)\cdot\big(\max_{i,j}|g_{ij}|_v\big)\cdot|P|_v^2\cdot\big(\max|f_i(P)|\big).\]
Now raise to the $n_v/|K:\mathbb{Q}|$ and multiply over all $v\in M_K$. Since the $g_{ij}$ are relatively simple, we get that \[H(P)^2\leq12\cdot\prod_{v\in M_K}\max\Big\{\Big\vert\frac{1}{4}\Big\vert_v,\Big\vert\frac{a}{4}\Big\vert_v,\;|a|_v,\;|b|_v\Big\}\cdot H\Big((x/z)\circ\phi_1(P)\Big).\] After taking the $\log$ of both sides, we obtain the desired lower bound on $h_E\big(\phi_1(P)\big)$.         
\end{proof}
Next, we prove the nonsingularity of the generalized Chebychev curve $X_{d,k}: T_d(x)+T_d(y)=k$, where $T_d$ is the Chebychev polynomial degree $d$. In particular, we can describe their primes of bad reduction.     
\begin{proposition} Let $K$ be a field of characteristic not dividing $d$ and let $X_{d,k}:T_d(x)+T_d(y)=k$. If $k\neq0$ and $k\neq\pm{4}$ in $K$, then $X_{d,k}$ is nonsingular. In particular, the Chebychev curve $X_d$ (corresponding to $k=1$) is nonsingular when $K$ is a number field. 
\end{proposition}{\label{nonsingularity}}
\begin{proof} Suppose that $(x,y)\in X_{d,k}(\bar{K})$ is a singular point and write $x=w+w^{-1}$ for some $w\in\bar{K}$. In particular,\begin{equation}{\label{chainrule}}0=T_d'(x)=\frac{w^2}{w^2-1}\cdot\frac{w^d-w^{-d}}{w}\cdot d,\end{equation} 
by the chain rule. In any case, $w^d=\pm{1}$. We can repeat this argument for $y$, deducing that $y=v+v^{-1}$ for some $v^d=\pm{1}$. In particular, \begin{equation}{\label{singularity}}k=T_d(x)+T_d(y)=w^d+w^{-d}+v^d+v^{-d}\in\{0,\pm{4}\}.\end{equation} 
\end{proof}

\end{document}